\documentclass{amsart}

\usepackage{amsfonts,amssymb,amsthm,amsmath}
\usepackage{xcolor}
\usepackage{listings}
\usepackage{hyperref}
\hypersetup{
    colorlinks=true,
    linkcolor=red!60!black,
    urlcolor=cyan!50!black,
    citecolor=green!50!black
    }

\usepackage[margin=1in]{geometry}

\newcommand{\bbC}{\mathbb{C}}
\newcommand{\bbN}{\mathbb{N}}
\newcommand{\bbP}{\mathbb{P}}
\newcommand{\bbZ}{\mathbb{Z}}

\newcommand{\Ann}{\mathrm{Ann}}
\newcommand{\HF}{h}
\newcommand{\rk}{\mathrm{rk}}

\newtheorem{theorem}{Theorem}[section]
\newtheorem{proposition}[theorem]{Proposition}
\newtheorem{lemma}[theorem]{Lemma}

\theoremstyle{definition}
\newtheorem{definition}[theorem]{Definition}

\newtheorem{question}[theorem]{Question}
\newtheorem{claim}[theorem]{Claim}

\theoremstyle{remark}
\newtheorem{remark}[theorem]{Remark}

\numberwithin{equation}{section}

\begin{document}

\title[Waring decompositions with different Hilbert functions]{Waring decompositions of special ternary forms\\ with different Hilbert functions}


\author{Elena Angelini}
\author{Luca Chiantini}
\address{Dip. di Ingegneria dell'Informazione e Scienze Matematiche, U. di Siena, Italy}
\curraddr{}
\email{elena.angelini@unisi.it, luca.chiantini@unisi.it}
\thanks{}

\author{Alessandro Oneto}
\address{Dip. di Mathematics, U. of Trento, Via Sommarive 14, 38123 Povo (Trento), Italy}
\curraddr{}
\email{alessandro.oneto@unitn.it}

\thanks{The authors are members of the Italian GNSAGA-INDAM}


\date{}

\begin{abstract}
  We prove the existence of ternary forms admitting apolar sets of points of cardinality equal to the Waring rank, but having different Hilbert function and different regularity. This is done exploiting liaison theory and Cayley-Bacharach properties for sets of points in the projective plane. 
\end{abstract}

\maketitle

\section{Introduction}
\subsection{Waring decompositions} Let $S = \bigoplus_{d\in\bbZ}S_d = \bbC[x_0,\ldots,x_n]$ be the standard graded polynomial ring in $n$ variables with complex coefficients.
\begin{definition}[Waring rank]\label{def:Waring}
  Given a degree-$d$ homogeneous polynomial $f \in S_d$, a \emph{Waring decomposition} of $f$ is an expression as sum of $d$-th powers of linear forms, i.e.,
\begin{equation}\label{eq:Waring_dec}
  f = \sum_{i=1}^r \ell_i^d, \quad \ell_i \in S_1.
\end{equation}
The minimal length $r$ of such a decomposition is the \emph{Waring rank}, denoted $\rk(f)$. 
\end{definition}
The study of Waring decompositions has its roots in the XIX century thanks to the works of Sylvester \cite{sylvester1851essay,sylvester1851remarkable}, Clebsch \cite{clebsch1861ueber}, Richmond \cite{richmond1904canonical}, Palatini \cite{palatini1902sulla,palatini1903sulla}, Terracini \cite{terracini1915sulla}, among many others. The question regained a lot of interest in the last decades. The celebrated theorem by Alexander and Hirschowitz in 1995 established the general Waring rank in any number of variables and degrees \cite{alexander1995polynomial}. The attention to Waring decompositions has been motivated also by its connection to \textit{tensor decompositions}; in this case, symmetric tensors. Indeed, tensor decompositions find a lot of uses in several applications, see for example \cite{kolda2009tensor,landsberg2012tensors}. The computation of most tensor decompositions is known to be an NP-hard problem and the same is conjectured fo symmetric tensors \cite{hillar2013most}. Therefore, a theoretical study of related algebraic and geometric structures can be useful to shed lights on them. 

\subsection{Apolar sets of points}
Waring decompositions have a geometric interpretation in terms of \emph{Veronese varieties}. Consider the \emph{Veronese embedding} given~by 
\[
  \nu_d : \bbP S_1 \longrightarrow \bbP S_d, \qquad [\ell] \mapsto [\ell^d].
\]
\begin{definition}[Apolar sets of points]
  Given $f \in S_d$ and a set of points $A \subset \bbP S_1$, we say that $A$ is \emph{apolar} to $f$ if $[f] \in \langle \nu_d(A) \rangle$. If there is no subset $A' \subset A$ such that $[f] \in \langle \nu_d(A') \rangle$ we say that $A$ is \emph{non-redundant}. If the cardinality of $A$ equals the Waring rank of $f$ we say that $A$ \emph{computes (the rank of)} $f$.
\end{definition}
Any Waring decomposition of a polynomial $f \in S_d$ has a geometric counterpart given by a set of points $A \subset \bbP S_1$ apolar to $f$, i.e., the equation \eqref{eq:Waring_dec} is equivalent to say that $[f] \in \langle \nu_d(A) \rangle$ for $A = \{[\ell_1],\ldots,[\ell_r]\} \subset \bbP S_1$. This geometric interpretation allows us to define a notion of rank with respect to any algebraic variety $X \subset \bbP^N$. I.e., the \textit{rank} of any $p \in \bbP^N$ with respect to $X$ is the minimal cardinality of a set of points $A \subset X$ such that $p \in \langle A \rangle$. 

\medskip
A more algebraic way to relate Waring decompositions with sets of points is through \emph{apolarity theory}. By Macaulay \cite{macaulay1916algebraic} (see also \cite[Lemma 2.14]{iarrobino1999power}), there is a one-to-one correspondence between artinian Gorenstein graded algebras with socle degree $d$ and homogeneous polynomials of degree $d$. Let $R = \bigoplus_{d\in\bbZ} R_d = \bbC[y_0,\ldots,y_n]$ be a second polynomial ring and consider the \textit{apolar action} of $R$ over $S$ by partial derivatives, i.e., for any $g \in R_e, f \in S_d$, we define $g \circ f := g(\partial_0,\ldots,\partial_n)f \in S_{d-e}$. Given $f \in S_d$, the \emph{apolar ideal} of $f$ is the annihilator of $f$ with respect to the apolar action, i.e.,
\[
  \Ann(f) := \{g \in R ~:~ g \circ f = 0\} \subset R.
\]
Macaulay's duality affirms that an artinian algebra $A = R/J$ is Gorenstein of socle degree $d$ if and only if $J = \Ann(f)$ for some $f \in R_d$. The \emph{Apolarity Lemma}, see \cite[Lemma 1.15]{iarrobino1999power}, affirms that Waring decompositions of a polynomial $f \in S_d$ correspond to sets of points $A \subset \bbP S_1$ such that $I(A) \subset \Ann(f)$, i.e., the equation \eqref{eq:Waring_dec} is equivalent to say that $I(A) \subset \Ann(f)$ for $A = \{[\ell_1],\ldots,[\ell_r]\} \subset \bbP S_1$. In other words, the study of Waring decompositions is equivalent to study ideals of sets of points in $\bbP S_1$, i.e., radical homogeneous ideals of height one, contained in artinian Gorenstein ideals. 

\subsection{Hilbert function of apolar sets of points} 
Most of algebraic invariants of algebraic varieties are encoded by the \emph{Hilbert function}. Given a homogeneous ideal $I \subset R$, let $I_d = I \cap R_d$ be its degree-$d$ homogeneous part. The \emph{Hilbert function} of $R/I$ is the numerical function $\HF_{R/I} : \bbZ \rightarrow \bbZ$ with 
\begin{equation}\label{eq:HF}
  \HF_{R/I}(d) := \dim_{\bbC}[R_d/I_d] = \dim_\bbC R_d - \dim_\bbC I_d
\end{equation}
If $A \subset \bbP S_1$ is a set of points, then we write $\HF_A$ for the Hilbert function of $R/I(A)$ where $I(A)$ is the defining homogeneous ideal of $A$. 
The study of Hilbert functions of reduced sets of points has been successfully used to compute Waring ranks. For example, the Hilbert function of the apolar ideal of a given polynomial provides a lower bound for the Waring rank of the polynomial. In some cases, these were shown to be optimal, see e.g. \cite{carlini2012solution,carlini2018symmetric}.

If a polynomial admits a unique minimal Waring decomposition, then we say that the polynomial is \emph{identifiable}. It is worth to mention that the general homogeneous polynomial with rank $r$ strictly smaller than the general one is identifiable except for the special cases $(n,d,r) \in \{(2,6,9),(3,4,8),(5,3,9)\}$, see \cite{chiantini2017generic}. In \cite{angelini2018identifiability,angelini2020identifiability,angelini2022minimality}, \emph{Cayley-Bacharach properties} and \emph{liaison theory} of sets of points in the projective plane were exploited to provide a criterion to certify that a form is identifiable. On the other hand, when the polynomial is not identifiable, it is interesting to study the family of apolar sets of points computing the rank of the polynomial, see \cite{ranestad2000varieties}. In \cite{mourrain2020minimal}, a classification of low rank polynomials is given in terms of the possible Hilbert functions of apolar sets of points computing the rank of the polynomials.

It is well-known that the Hilbert function of a set of points is strictly increasing until it reaches its cardinality. We call \emph{regularity} of $A$ the smallest integer $i$ for which the Hilbert function of $A$ stabilizes, i.e., such that $\HF_A(i) = \HF_A(i+1)$.

\subsection{The question.}
In this short note, we address the following question. 
\begin{question}[{\cite[Question 2.11]{mourrain2020minimal}}]\label{question}
  Let $f \in S_d$. Do all sets of points computing the rank of $f$ share the same regularity? More generally, do they share the same Hilbert function?
\end{question}
The answer is known to be positive for binary forms ($n = 2$), because all sets of points with same cardinality on the projective line share the same Hilbert function, and for monomials $x_0^{d_0}\cdots x_n^{d_n}$ because all sets of points computing the rank are known to be complete intersections with generators of degrees $d_1+1,\ldots,d_n+1$ if $d_0 \leq d_i$ for all $i$, see \cite{buczynska2013waring}. In the present note, we prove that the answer to Question~\ref{question} is negative for ternary forms ($n=3$) if the degree is sufficiently large. We provide two cases. In Section \ref{ssec:example_1}, we consider degree-$10$ ternary forms having rank $22$, that is the general rank. Among them, we consider the special ones which admit a unique quintic in the apolar ideal. We show in Theorem \ref{thm:example_1} that the general such form has a decomposition lying on the quintic and one not contained in the quintic. Along with this proof, we notice that there exist degree-$10$ ternary forms of rank $22$ and with a unique quintic in the apolar ideal which have higher rank with respect to such quintic curve, see Remark \ref{rmk:remark_22}. We recall that the study of the connection between catalecticant matrices and the Hilbert function of certain decompositions was also used in \cite{chiantini2021footnote} where a complete stratification of ternary sextics of rank nine. In Section \ref{ssec:example_2}, we consider degree-$13$ ternary forms having rank $30$, that is strictly subgeneric. In Theorem \ref{thm:example_2}, we show that there exist two sets of $30$ points which compute the rank of a degree-$13$ ternary form and have different regularity. Our proofs follow the ideas of \cite{angelini2018identifiability,angelini2020identifiability,angelini2022minimality}, by exploiting properties of Hilbert functions of points in the projective plane and, in particular, \emph{liaison theory}. In Section \ref{sec:M2}, we show an explicit construction with the support of the algebra software Macaulay2 \cite{M2}.

\section{Algebraic properties of points in $\bbP^n$}

\subsection{Hilbert function of points in $\bbP^n$}
Let $A \subset \bbP^n$ be a set of points of cardinality $\ell(A)$. We recalled in \eqref{eq:HF} the definition of the \textit{Hilbert function} of $A$. The \textit{first difference} is the numerical function given by
\[
  Dh_A : \bbZ \rightarrow \bbZ, \quad D\HF_A(i) := \HF_A(i)-\HF_A(i-1).
\]
We recall some elementary properties of Hilbert functions of points in $\bbP^n$ that will be used later. We refer to \cite{chiantini2019hilbert} for an extensive exposition including proofs.
\begin{lemma}\cite[Lemma 2.16, Proposition 2.17, Proposition 2.18]{chiantini2019hilbert}\label{lemma:HF_properties}
   Let $A \subset \bbP^n$ be a set of points of cardinality $\ell(A)$.
  \begin{enumerate}
    \item $ h_{A}(j) \leq \ell(A)$, for all $j$;
    \item $ Dh_{A}(j) = 0$, for $ j < 0$;
    \item $ h_{A}(0) = Dh_{A}(0) = 1$;
    \item $ Dh_{A}(j) \geq 0 $, for all $ j $;
    \item $ h_{A}(i) = \sum_{j=0}^i Dh_{A}(j) $; 
    \item $ h_{A}(j) = \ell(A) $, for all $ j \gg 0$; 
    \item $ Dh_{A}(j) = 0$, for $ j \gg 0 $;
    \item $ \sum_{j=0}^\infty Dh_{A}(j) = \ell(A) $;
    \item for any $A' \subset A$, $h_{A'}(j) \leq h_{A}(j)$ and $Dh_{A'}(j) \leq Dh_{A}(j)$ for any $j$;
    \item if $Dh_A(j) < j$, then $Dh_A(j) \geq Dh_A(j+1)$.
    \end{enumerate}
\end{lemma}
Finally, we denote
\[
  h^1_A(d) := \ell(A) - \HF_A(d) = \sum_{i=d+1}^\infty Dh_A(i). 
\] 
The following proposition allows us to understand if, given two sets of points, there exists a homogeneous polynomial of given degree to which they are both apolar.
\begin{proposition}\cite[Proposition 2.19]{angelini2020identifiability}\label{prop:dim_intersection}
  Let $A,B \subset \bbP^n$ be finite sets of points and let $Z = A \cup B$. Then, for any $d \in \bbN$,
  \[
    \dim \left(\langle \nu_d(A) \rangle \cap \langle \nu_d(B) \rangle\right) = \ell(A \cap B) - 1 + h^1_Z(d).
  \]
\end{proposition}

\subsection{Liaison for points in $\bbP^2$}
We describe here some construction that will be used in Section \ref{sec:examples} to construct the examples replying to Question \ref{question}. These are basic constructions from \emph{liaison theory}; we refer to \cite{peskine1974liaison,migliore1998introduction} for a complete presentation of the theory. 
\begin{definition}
  Let $A,B \subset \bbP^2$ be two sets of points. We say that \emph{$A$ and $B$ are linked} if there exists a complete intersection $Z$ such that $I_B = I_Z : I_A$. When $A \cap B = \emptyset$, this is equivalent to say that $Z = A \cup B$. We say that $B$ is the \emph{residue} of $A$ with respect to $Z$. 
\end{definition}
A set of points in $\bbP^2$ is arithmetically Cohen-Macaulay of codimension $2$, hence a minimal free resolution is provided by the \emph{Hilbert-Burch Theorem}, see \cite[Section~3A]{eisenbud2005geometry}. If $A$ and $B$ are linked, then a Hilbert-Burch matrix of $B$ can be recovered from one of $A$ by the \emph{mapping cone} construction, see \cite[Proposition 5.2.10]{migliore1998introduction}.

Let $S(i)$ be the polynomial ring with the gradation shifted by $i$, i.e., $S(d)_i := S_{d+i}$. A similar notation is used for every graded $S$-module. If $A$ and $B$ are linked through a complete intersection $Z$ of type $(d_1,d_2)$, then we have the following free resolutions
\[
  \begin{array}{c c c c c c c c c}
    0 & \rightarrow & S(-d_1-d_2) & \rightarrow & S(-d_1) \oplus S(-d_2) & \rightarrow & I_Z & \rightarrow & 0 \\
    & & \downarrow & & \downarrow\phi & & \downarrow \\
    0 & \longrightarrow & F_1 & \xrightarrow{M} & F_0 & \longrightarrow & I_A & \longrightarrow & 0 \\
  \end{array}
\]
where $\phi : S(-d_1) \oplus S(-d_2) \rightarrow F_0$ is the map induced by the right-most vertical arrow which corresponds to the inclusion $I_Z \subset I_A$ and $M$ is a Hilbert-Burch matrix of $A$. Then, by the mapping cone construction, a resolution of $I_B$ is given by
\[
  \begin{array}{c c c c c c c c c}
    0 & \rightarrow & F_0(-d_1-d_2)^\vee & \xrightarrow{(\phi | M)^t} & \begin{array}{c} S(-d_1) \oplus S(-d_2) \\ \oplus \\ F_1(-d_1-d_2)^\vee\end{array} & \rightarrow & I_B & \rightarrow & 0.
  \end{array}
\]
The Hilbert functions between the two linked sets of points are linked by the following formula on their first differences:
\begin{equation}\label{eq:sum_Dh}
  Dh_B(i) + Dh_A(d_1+d_2-i-2) = Dh_Z(i),
\end{equation}
see \cite[Theorem 3]{davis1985gorenstein} or \cite[Corollary 5.2.19]{migliore1998introduction}.

\subsection{Cayley-Bacharach properties}
We recall here also some basic fact about Cayley-Bacharach properties for sets of points in projective space. They are used to understand the possible Hilbert functions of sets of points $A \cup B$ where $A$ and $B$ are apolar to the same homogeneous form. These properties have been crucial to provide certificates for identifiability of forms in \cite{angelini2018identifiability,angelini2020identifiability,angelini2022minimality}.

\begin{definition}
  Let $Z \subset \bbP^n$ be a set of points. We say that it satisfies the \textit{Cayley-Bacharach property in degree $d$} ($CB(d)$) if, for all $p \in Z$, it holds that every form of degree $d$ vanishing at $Z \smallsetminus \{p\}$ also vanishes at $p$.
\end{definition}

\begin{proposition}\cite[Proposition 2.25]{angelini2020identifiability}\label{prop:CB}
  Let $f \in S_d$ and let $A$ be a non-redundant set of points apolar to $f$. Let $B$ a second non-redundant set of points apolar to $f$ such that $A \cap B = \emptyset$. Then, the set $Z = A \cup B$ satisfies $CB(d)$. 
\end{proposition}

\begin{theorem}\cite[Theorem 4.9]{angelini2018identifiability}\label{thm:CB_implies}
  Let $Z \subset \bbP^n$ be a set of points satisfying $CB(d)$. Then, for any $i \in \{0,\ldots,d+1\}$,
  \[
    \sum_{j=0}^i Dh_Z(j) \leq \sum_{j=0}^i Dh_Z(d+1-j). 
  \]
\end{theorem}

\subsection{Secant varieties and Terraccini loci}\label{sec:secant}
The geometric framework to deal with Waring decompositions is the one of \textit{secant varieties} of Veronese varieties, see \cite{carlini2014four,bernardi2018hitchhiker}. Given an irreducible algebraic variety $X \subset \bbP^N$, the \textit{$r$-th secant variety} is an irreducible variety defined as the Zariski closure of the union of linear spaces spanned by $r$-tuples of points on $X$, i.e.,
\[
  \sigma_r(X) := \overline{\bigcup_{p_1,\ldots,p_r \in X} \langle p_1,\ldots,p_r \rangle}\subset \bbP^N.
\]
If $X$ is the $d$-th Veronese embedding of $\bbP S_1$, then, by Definition \ref{def:Waring},
\[
  \sigma_r(\nu_d(\bbP^n)) = \overline{\{[f] \in \bbP S_d ~:~ \rk(f) \leq r\}}\subset\bbP S_d.
\]
Secant varieties are clearly nested into each other by definition, but, if the variety $X$ is non-degenerate, then we have that its secant varieties eventually fill the ambient space, see \cite[Exercise 2.6]{carlini2014four}. For example, the minimal $r$ such that $\sigma_r(\nu_d(\bbP^n)) = \bbP S_d$ is the Waring rank of the general degree-$d$ form in $n+1$ variables. 

Understanding which secant variety fills the ambient space reduces to the computation of dimensions of secant varieties. By a simple parameter count, the \textit{expected dimension} of the $r$-th secant variety of an algebraic variety $X \subset \bbP^N$ is
\[
  {\rm exp}.\dim\sigma_r(X) := \min\{r\dim(X) + r - 1, N\}.
\]
The expected dimension is always an upper bound for the actual dimension and if $\dim \sigma_r(X) < {\rm exp}.\dim\sigma_r(X)$ then we say that $X$ is \textit{$r$-defective}. The classification of defective varieties is a very classical problem with roots in the late XIX century. By Palatini \cite{palatini1909sulle}, if $\sigma_r(X)$ is not an hypersurface and does not fill the ambient space, then it is at least of codimension $2$ inside $\sigma_{r+1}(X)$. As an immediate corollary, projective curves are never defectives. A classification of defective surfaces is due to Severi and Terracini, see \cite[Theorem 14 and Theorem 15]{chiantini2004lectures}, while the classification of defective threefolds is due to Scorza \cite{scorza1908determinazione}. The case of fourfolds has been treated first by Scorza in \cite{scorza1909} (see also \cite{scorza1960}) and then completed in the recent preprint \cite{chiantini2020secant}.

If we denote by $X^{(r)}$ the $r$-th symmetric product of the variety $X$, the \textit{$r$-th abstract secant variety} of $X$ is 
\[
  ab\sigma_r(X) := \overline{\{((p_1,\ldots,p_r),q) ~:~ q \in \langle p_1,\ldots,p_r \rangle\}} \subset X^{(r)} \times \bbP^N.
\]
Throughout the paper, we always denote by $\pi_i$ the projection onto the $i$-th factor from a cartesian product. The $r$-th abstract secant variety is irreducible of dimension $r\dim(X)+r-1$. Hence, if $\sigma_r(X) \subsetneq \bbP^N$, non-defectiveness is equivalent to say that the projection $\pi_2 : ab\sigma_r(X) \rightarrow \bbP^N$ has generically finite fibers. The \textit{$r$-th Terracini locus} of $X$ is
\[
  \mathbb{T}_r(X) := \overline{\left\{(p_1,\ldots,p_r) ~:~ \substack{p_i \in X^{\rm smooth}, \{p_1,\ldots,p_r\} \text{ are linearly independent}, \\ T_{p_1}(X),\ldots,T_{p_r}(X) \text{ are linearly dependent}}\right\}} \subset X^{(r)}.
\]
The main tool to study dimensions of secant varieties is the \textit{Terracini's Lemma} which describes the general tangent space to secant varieties, see \cite[Lemma 1]{bernardi2018hitchhiker}. In particular, it claims that, given general points $p_1,\ldots,p_r \in X$ and a general point $p \in \langle p_1,\ldots,p_r \rangle$, then
\[
  T_p\sigma_r(X) = \langle T_{p_1}(X),\ldots,T_{p_r}(X) \rangle. 
\]
If $X$ is not $r$-defective and $\sigma_r(X)$ does not fill the space, then $\mathbb{T}_r(X)$ is a proper subvariety of $X^{(r)}$. In other words, the preimage of the Terracini's locus $\pi_1^{-1}(\mathbb{T}_r(X))$ is contained in the locus where the differential of the projection $\pi_2 : ab\sigma_r(X) \rightarrow \bbP^N$ drops rank. 

The complete classification of defective Veronese varieties is known as the \textit{Alexander-Hirschowitz Theorem} \cite{alexander1995polynomial}. In the 1980s, Alexander and Hirschowitz defined a very powerful machinery to deal with problems on polynomial interpolation such as the computation of the dimension of the linear system of degree-$d$ hypersurfaces of $\bbP^n$ having $r$ singularities at general points. The latter problem is equivalent to compute dimensions of secant varieties of the Veronese varieties and, a posteriori, to compute the general Waring rank in $S_d$.
\begin{theorem}[{Alexander-Hirschowitz Theorem, \cite{alexander1995polynomial}}]\label{thm:ah}
  The $r$-th secant variety of the Veronese variety $\nu_d(\bbP^n)$ is non-defective except for 
  \begin{itemize}
    \item $d = 2$, $n \geq 2$ and $2 \leq r \leq n$;
    \item $d = 3, n = 4, r = 7$;
    \item $d = 4, (n,r) \in \{(2,5), (3,9), (4,14)\}$.
  \end{itemize}
\end{theorem}

\section{The examples}\label{sec:examples}
In this section, we give a negative answer to Question \ref{question} by showing that there exist ternary forms admitting apolar sets of points computing their Waring rank with different Hilbert function as well as different regularity.

\medskip
The idea is to start from a special set of points $Z_1$ and, by means of liaison theory, to construct a second set of points of the same cardinality $Z_2$, but different Hilbert function, such that $\langle \nu_d(Z_1) \rangle \cap \langle \nu_d(Z_2) \rangle$ is non-empty, i.e., there is a degree-$d$ form $f$ whose $Z_1$ and $Z_2$ are both apolar to. After the construction, we prove that indeed $Z_1$ and $Z_2$ compute the rank of $f$. 

\subsection{First example. Degree-$10$ ternary forms of rank $22$.}\label{ssec:example_1}
\begin{theorem}\label{thm:example_1}
  There exists a form $f \in S_{10}$ of rank $22$ having two sets of points computing its rank and presenting different Hilbert functions. 
\end{theorem}
\begin{proof}
  Let $Z_1 \subset \bbP^2$ be a set of $22$ general points on a quintic curve. Then,
  \begin{equation}\label{eq:deg10_Z1}
    \begin{matrix}
      h_{Z_1} : & 1 & 3 & 6 & 10 & 15 & 20 & 22 & 22 & \cdots \\
      D h_{Z_1} : & 1 & 2 & 3 & 4 & 5 & 5 & 2 & - 
    \end{matrix}
  \end{equation}
  Let $L \subset \bbP^2$ be a general line and let $Y$ be a set of five general points on $L$. Consider $A = Z_1 \cup Y$. Clearly there are no quintics through $A$ since there is a unique quintic through $Z_1$ and the line $L$ is chosen generically. Moreover, $Y$ imposes independent conditions on the space of sextics through $Z_1$: indeed, it is enough to observe that $\dim I(Z_1)_6 - \dim I(Z_1)_5 = 5 \geq \#Y$, see e.g. \cite[Lemma 1.9-(i)]{catalisano2007segre}. Hence,
  \begin{equation}\label{eq:deg10_A}
    \begin{matrix}
      h_{A} : & 1 & 3 & 6 & 10 & 15 & 21 & 27 & 27 & \cdots \\
      D h_{A} : & 1 & 2 & 3 & 4 & 5 & 6 & 6 & - 
    \end{matrix}
  \end{equation}
  Let $X$ be a complete intersection of two general septics containing $A$. Then, we define $Z_2$ to be the residue of $A$ through $X$. Hence, by \eqref{eq:sum_Dh}, we have
  \begin{equation}\label{eq:deg10_Z2}
    \begin{array}{*{21}c}
      &  & $\tiny{0}$ & $\tiny{1}$ & $\tiny{2}$ & $\tiny{3}$ & $\tiny{4}$ & $\tiny{5}$ & $\tiny{6}$ & $\tiny{7}$ & $\tiny{8}$ & $\tiny{9}$ & $\tiny{10}$ & $\tiny{11}$ & $\tiny{12}$ \\
      Dh_{X}(i) & : & 1 & 2 & 3 & 4 & 5 & 6 & 7 & 6 & 5 & 4 & 3 & 2 & 1 \\
      Dh_{Z_2}(i) & : & 1 & 2 & 3 & 4 & 5 & 6 & 1 & - & - & - \\
      Dh_{A}(12-i) & : & & & & & & & 6 & 6 & 5 & 4 & 3 & 2 & 1 & 
    \end{array}
  \end{equation}
  Note that, since we are considering a linkage through a complete intersection of two curves both of degree at least the regularity of $Z_1$ plus one, which is seven, then, by the Peskine-Szpiro Theorem, see \cite[Theorem 6.4.4 and Example~6.4.5]{migliore1998introduction}, $Z_2$ is smooth and $Z_1 \cap Z_2 = \emptyset$. Hence, $Z_2$ is a set of $22$ distinct points disjoint from $Z_1$ and with Hilbert function different than $Z_1$. 
  \begin{claim}\label{claim:deg10_a}
    $Z_1$ and $Z_2$ are apolar to a common degree-$10$ ternary form.
  \end{claim}
  \begin{proof}
    By Proposition \ref{prop:dim_intersection}, it is enough to compute the Hilbert function of the union $Z_1 \cup Z_2$. Note that, by construction, $Z_1 \cup Z_2$ is linked to the set $Y$ of five collinear points. Hence, by \eqref{eq:sum_Dh}
    \[
      \begin{array}{*{21}c}
        &  & $\tiny{0}$ & $\tiny{1}$ & $\tiny{2}$ & $\tiny{3}$ & $\tiny{4}$ & $\tiny{5}$ & $\tiny{6}$ & $\tiny{7}$ & $\tiny{8}$ & $\tiny{9}$ & $\tiny{10}$ & $\tiny{11}$ & $\tiny{12}$ \\
        Dh_{X}(i) & : & 1 & 2 & 3 & 4 & 5 & 6 & 7 & 6 & 5 & 4 & 3 & 2 & 1 \\
        Dh_{Z_1 \cup Z_2}(i) & : & 1 & 2 & 3 & 4 & 5 & 6 & 7 & 6 & 4 & 3 & 2 & 1 \\
        Dh_{Y}(12-i) & : & & & & & & & & & 1 & 1 & 1 & 1 & 1 & 
      \end{array}
    \]
    Hence, by Proposition \ref{prop:dim_intersection}, $\dim \langle \nu_{10}(Z_1) \rangle \cap \langle \nu_{10}(Z_2) \rangle = 0$. I.e., there exists $f \in S_{10}$ such that $[f] \in \langle \nu_{10}(Z_1) \rangle \cap \langle \nu_{10}(Z_2) \rangle$.
  \end{proof}
  Now, we need to show that the general degree-$10$ form $f$ constructed as in Claim~\ref{claim:deg10_a} is indeed of rank $22$. Let $\mathcal{X}$ be the variety of degree-$10$ forms admitting an apolar set of $22$ reduced points lying on a quintic. 
  \begin{claim}\label{claim:deg10_b}
    The variety $\mathcal{X}$ is irreducible and has codimension $2$. 
  \end{claim}
  \begin{proof}
    Note that the variety of degree-$10$ ternary forms $f\in S_{10}$ admitting a quintic in the apolar ideal is a hypersurface defined by the vanishing of the determinant of the central \textit{catalecticant matrix} associated to the map 
    \[{\rm cat}_5(f) : R_5 \rightarrow S_5, g \mapsto g \circ f.\]
    The fact that the variety $\{\det({\rm cat}_5(f)) = 0\}$ is irreducible follows from \cite{eisenbud1988linear}. Fixed a plane quintic $C \subset \bbP^2$, by Riemann-Roch Theorem, we have that $\langle \nu_{10}(C) \rangle = \bbP^{44}$. Since curves are non-defective, $\dim \sigma_{22}\nu_{10}(C) = 43$, see Section \ref{sec:secant}. Therefore the general form $f \in \langle \nu_{10}(C) \rangle$ has rank $23$ with respect to $\nu_{10}(X)$ and $\mathcal{X}$ has codimension one in the hypersurface of forms admitting a quintic in the apolar ideal. In other words, if we consider the incidence variety of forms lying on the $22$-nd secant variety of the degree-$10$ Veronese embedding of a plane quintic \[\mathcal{I} = \{(C,[f]) ~:~ [f] \in \sigma_{22}(\nu_{10}(C))\} \subset \bbP S_5 \times \bbP S_{10},\] then $\mathcal{X} = \pi_2(\mathcal{I})$. The projection $\pi_2$ is generically one-to-one because the quintic in the apolar ideal for a general form $[f] \in \mathcal{X}$ is uniquely the determinant of ${\rm cat}_5(f)$. Given a quintic $C$, we have that $\pi_1^{-1}(C) \cong \sigma_{22}(\nu_{10}(C))$. Hence, $\dim\mathcal{X} = \dim\mathcal{I} = \dim \bbP S_5 + \dim \sigma_{22}(\nu_{10}(C)) = 20 + 43 = 63$ inside $\bbP S_{10} = \bbP^{65}$. 
  \end{proof}

  \begin{claim}\label{claim:deg10_22}
    The general form $[f] \in \mathcal{X}$ has rank $22$.
  \end{claim}
  \begin{proof}
    By Alexander-Hirschowitz Theorem (Theorem \ref{thm:ah}), 
    \[
      \dim\sigma_{21}(\nu_{10}(\bbP^2)) = 21 \cdot 2 + 20 = 62,
    \] 
    i.e., $\sigma_{21}(\nu_{10}(\bbP^2))$ has codimension $3$. Hence, by Claim \ref{claim:deg10_b}, $\mathcal{X} \cap \sigma_{21}(\nu_{10}(\bbP^2))$ is a proper subvariety of~$\mathcal{X}$.
  \end{proof}
  
  Consider $\mathcal{Y} \subset \mathcal{X}$ the set of degree-$10$ forms that are constructed as in Claim~\ref{claim:deg10_a}. Let ${\rm Hilb}^\circ_{22}(\bbP^2)$ be the open subset of the Hilbert scheme of $22$ points in $\bbP^2$ and let ${\rm Hilb}^\circ_{22,5}(\bbP^2)$ denote the set of all sets of $22$ general points lying on a quintic of $\bbP^2$. Then, consider the incidence variety
    \[
      \mathcal{J} = \{(Z_1,Z_2,[f]) ~:~ \langle \nu_{10}(Z_1) \rangle \cap \langle \nu_{10}(Z_2) \rangle = [f]\} \subset {\rm Hilb}^\circ_{22,5}(\bbP^2) \times {\rm Hilb}^\circ_{22}(\bbP^2) \times \bbP S_{10}.
    \]
  By construction, $\mathcal{Y} = \pi_3(\mathcal{J})$. Moreover, from the construction described at the beginning of the proof, the projection $\pi_1$ is dominant, i.e., for a general choice of a set $Z_1$ of $22$ points on a quintic in $\bbP^2$, there exists $Z_2 \in {\rm Hilb}^\circ_{22}(\bbP^2)$ and $f \in S_{10}$ such that $(Z_1,Z_2,[f]) \in \mathcal{J}$. With a similar construction, we can also show that the projection $\pi_2$ is dominant. 
  \begin{claim}\label{claim:deg10_c}
    The projection $\pi_2 : \mathcal{J} \rightarrow {\rm Hilb}^\circ_{22}(\bbP^2)$ is dominant.
  \end{claim}
  \begin{proof}
    Let $Z_2$ be a general set of $22$ points in $\mathbb{P}^2$. Take a set $Y$ of three general points collinear on a general line $L$ and let $A = Z_2 \cup Y$. Since there are no quintics through $Z_2$, then the three points impose independent conditions on sextics through $Z_2$, see \cite[Lemma 1.9]{catalisano2007segre}. Hence, 
    \[
      \begin{matrix}
        h_{A} : & 1 & 3 & 6 & 10 & 15 & 21 & 25 & 25 & \cdots \\
        D h_A : & 1 & 2 & 3 & 4 & 5 & 6 & 4 & - 
      \end{matrix}
    \] 
    Let $\{G_1,G_2,G_3\}$ be a basis for the space of sextics passing through $A$, i.e., $[I_A]_6 = \langle G_1,G_2,G_3 \rangle$. The restriction of $[I_A]_6$ on $L$ defines a non-complete linear series $g^2_3$ away from the set of points $Y$. 
    The linear series on $L$ is generically base-point free, then it defines a regular map on $\bbP^2$ which maps $L$ into a singular cubic curve $C$. Indeed, since being base-point free is an open condition, this can be checked directly on a specific example with the support of the algebra software Macaulay2~\cite{M2}, see Section \ref{sec:M2_1}. For example, without loss of genericity, let $L = \{x_0 = 0\}$ and $Y = \{(0:1:0), (0:0:1), (0:1:1)\}$. Hence, for $i = 1,2,3$, the restriction of $G_i$ on $L$ is given by $G_i(0,x_1,x_2) = x_1x_2(x_1-x_2)H_i(x_1,x_2)$ and we consider the map $\varphi : L \rightarrow \bbP^2, p \mapsto (H_1(p):H_2(p):H_3(p))$. Let $C = \varphi(L)$. The general cubic $C$ constructed as above is nodal. Indeed, since being nodal is an open condition on the space of rational plane cubics, it is enough to exhibit one example in which the above construction provides a nodal curve. This is checked with the support of the algebra software Macaulay2 \cite{M2}, see Section \ref{sec:M2_1}. 
    
    Now, let $X$ be the set of two distinct points on $L$ which are pre-images of the singular point of $C$ and consider the set of five points $X \cup Y \subset L$. By construction, $X \cup Y$ does not impose independent conditions on the space of sextics passing through $A$. This implies that
    \[
      \begin{array}{*{21}c}
        h_{X \cup A} & : & 1 & 3 & 6 & 10 & 15 & 21 & 26 & 27 & 27 & \cdots \\
        D h_{X \cup A} & : & 1 & 2 & 3 & 4 & 5 & 6 & 5 & 1 & - 
      \end{array}
    \]  
    Now, we consider the linkage to $X \cup A$ via a general complete intersection of type $(7,7)$, i.e., $U = (X \cup A) \cup Z_1$. Then, $Z_1$ is a set of $22$ points with the following Hilbert function
    \begin{equation}\label{eq:deg10_Z2}
      \begin{array}{*{21}c}
        &  & $\tiny{0}$ & $\tiny{1}$ & $\tiny{2}$ & $\tiny{3}$ & $\tiny{4}$ & $\tiny{5}$ & $\tiny{6}$ & $\tiny{7}$ & $\tiny{8}$ & $\tiny{9}$ & $\tiny{10}$ & $\tiny{11}$ & $\tiny{12}$ \\
        Dh_{U}(i) & : & 1 & 2 & 3 & 4 & 5 & 6 & 7 & 6 & 5 & 4 & 3 & 2 & 1 \\
        Dh_{Z_1}(i) & : & 1 & 2 & 3 & 4 & 5 & 5 & 2 & - & - & - \\
        Dh_{X\cup A}(12-i) & : & & & & & & 1 & 5 & 6 & 5 & 4 & 3 & 2 & 1 & 
      \end{array}
    \end{equation}
    Again, by Peskine-Szpiro Theorem, see \cite[Theorem 6.4.4 and Example~6.4.5]{migliore1998introduction}), $Z_1$ is smooth and $Z_1 \cap Z_2 = \emptyset$. Hence, $Z_1 \in {\rm Hilb}^\circ_{22,5}(\bbP^2)$. Moreover, by construction, $U = (Z_1 \cup Z_2) \cup (X \cup Y)$. I.e., the union $Z_1 \cup Z_2$ is linked through $U$ to the set $X \cup Y$ of $5$ collinear points. Hence, by \eqref{eq:sum_Dh},
    \[
      \begin{array}{*{16}c}
        & & $\tiny{0}$ & $\tiny{1}$ & $\tiny{2}$ & $\tiny{3}$ & $\tiny{4}$ & $\tiny{5}$ & $\tiny{6}$ & $\tiny{7}$ & $\tiny{8}$ & $\tiny{9}$ & $\tiny{10}$ & $\tiny{11}$ & $\tiny{12}$  \\
        Dh_U(i) & :     & 1 & 2 & 3 & 4 & 5 & 6 & 7 & 6 & 5 & 4 & 3 & 2 & 1 &  \\
        Dh_{X \cup Y}(12-i) & :   & & & & & & & & & 1 & 1 & 1 & 1 & 1 \\
        Dh_{Z_1\cup Z_2}(i) & : & 1 & 2 & 3 & 4 & 5 & 6 & 7 & 6 & 4 & 3 & 2 & 1 &  
      \end{array}
    \]
      Hence, $Dh_{Z_1\cup Z_2}(11) = 1$ and, by Proposition \ref{prop:dim_intersection}, there exists $f \in S_{10}$ such that $(Z_1,Z_2,[f]) \in \mathcal{J}$.
  \end{proof}
  \begin{claim}
    $\mathcal{Y}$ is dense in $\mathcal{X}$. In particular, the general $[f] \in \mathcal{Y}$ has rank $22$.
  \end{claim}
  \begin{proof}
    Since the $22$-nd secant variety of $\nu_{10}(\bbP^2)$ and of $\nu_{10}(C)$, for any quintic $C \subset \bbP^2$, are non-defective, we have that the Terracini loci in ${\rm Hilb}^\circ_{22,5}(\bbP^2)$ and ${\rm Hilb}^\circ_{22}(\bbP^2)$ are proper sub-varieties. In particular, we can restrict the projection $\pi_3$ to the dense subspace of $\mathcal{J}$ given by the triplets $(Z_1,Z_2,[f])$ such that $Z_1$ and $Z_2$ do not belong to the corresponding Terracini locus. This implies that, for the general $[f] \in \mathcal{Y}$, the fiber $\pi_3^{-1}([f])$ is finite, see \cite[Proposition 6.2]{ballicochiantini2021}. In particular, $\dim \mathcal{Y} = \dim \mathcal{J}$. 
    
    Since $\pi_2$ is dominant, for a general $Z_2 \in {\rm Hilb}^\circ_{22}(\bbP^2)$, we have that $\dim \mathcal{J} = \dim {\rm Hilb}^\circ_{22}(\bbP^2) + \dim \pi_2^{-1}(Z_2)$. Now, $\dim {\rm Hilb}^\circ_{22}(\bbP^2) = 44$. Following the construction in Claim \ref{claim:deg10_c}, we see that $\dim \pi_2^{-1}(Z_2) = 19$. Indeed, it depends on the choice of three collinear points, that is five parameters ($2$ to fix a line and $3$ to fix three points on it), and the choice of the general complete intersection of two septics, that is the choice of a pencil in the $8$-dimensional projective space of septics through the set of $27$ points $A$, i.e., other $14 = 2(9-2)$ parameters. In particular, $\dim \mathcal{J} = 63 = \dim \mathcal{X}$. Since $\mathcal{X}$ is irreducible, we conclude the first part of the statement. The second part of the statement follows directly from Claim \ref{claim:deg10_22}.
  \end{proof}
  Therefore, the general element of $\mathcal{Y}$ evinces the statement of the theorem.
\end{proof}
\begin{remark}\label{rmk:remark_22}
  In the proof of Claim \ref{claim:deg10_b} we observed that, given a general quintic curve $C \subset \bbP^2$, the general degree-$10$ form in $\langle \nu_{10}(C) \rangle = \bbP^{44}$ has rank $23$ with respect to the curve. However, since the general rank of degree-$10$ ternary forms is equal to $22$ by Alexander-Hirschowitz Theorem (Theorem \ref{thm:ah}), we also deduce that the general form in $\langle \nu_{10}(C) \rangle$ has rank $22$ with respect to the Veronese variety $\nu_{10}(\bbP^2)$. In other words, there exist degree-$10$ ternary forms whose apolar ideal has minimal generators in degree five but minimal Waring decompositions are given by sets of points whose ideal does not involve such minimal generators. 
\end{remark}

\subsection{Second example. Degree-$13$ ternary forms of rank $30$.}\label{ssec:example_2}

\begin{theorem}\label{thm:example_2}
  There exists a form $f \in S_{13}$ of rank $30$ having two sets of points computing its rank and admitting different Hilbert functions and regularities. 
\end{theorem}
\begin{proof}
Let $A$ be a set of $12$ general points lying on a cubic. I.e.,
\[
  \begin{matrix}
    h_A : & 1 & 3 & 6 & 9 & 12 & 12 & \cdots \\
    Dh_A : & 1 & 2 & 3 & 3 & 3 & - 
  \end{matrix}
\]
Then, we consider the set $Z_1$ of $30$ points linked to $A$ through a complete intersection $X = A \cup Z_1$ of $42$ points defined by a general sextic and a septic. I.e., by \eqref{eq:sum_Dh},
\[
  \begin{array}{*{15}c}
    Dh_X :     & 1 & 2 & 3 & 4 & 5 & 6 & 6 & 5 & 4 & 3 & 2 & 1 & - \\
    Dh_{Z_1} : & 1 & 2 & 3 & 4 & 5 & 6 & 6 & 2 & 1 & - 
  \end{array}
\]
We link $Z_1$ to another set $Z_2$ of $30$ points through a complete intersection $Y = Z_1 \cup Z_2$ of $60$ points defined by a sextic and degree-$10$ form. I.e., by \eqref{eq:sum_Dh},
\begin{equation}\label{eq:HF_example_deg13}
  \begin{array}{*{20}c}
    Dh_Y :     & 1 & 2 & 3 & 4 & 5 & 6 & 6 & 6 & 6 & 6 & 5 & 4 & 3 & 2 & 1 & - \\
    Dh_{Z_1} : & 1 & 2 & 3 & 4 & 5 & 6 & 6 & 2 & 1 & - \\
    Dh_{Z_2} : & 1 & 2 & 3 & 4 & 5 & 6 & 5 & 4 & - 
  \end{array}
\end{equation}
We immediately observe that $Z_1$ and $Z_2$ are both apolar to a form of degree $13$. Indeed, from the first difference of $Y = Z_1 \cup Z_2$ in \eqref{eq:HF_example_deg13}, we have $h^1_{14}(Z_1 \cup Z_2) = 1$. By Proposition \ref{prop:dim_intersection}, $\langle \nu_{13}(Z_1) \rangle \cap \langle \nu_{13}(Z_2) \rangle = [f]$ for some degree $13$ form $f$. 
\begin{claim}\label{claim:deg13}
  $f$ has rank $30$.
\end{claim}
\begin{proof}
  Let $Z_1$ be as above and $Z' \subset \bbP^n$ be another set of points apolar to~$f$. 
  
  First, assume that $Z_1 \cap Z' = \emptyset$ and let $\overline{Z} = Z_1 \cup Z'$. By Lemma~\ref{lemma:HF_properties}-(9) since $\sum_{i=0}^5 Dh_{Z_1}(i) = 21$ then also $\sum_{i=0}^5 Dh_{\overline{Z}}(i) = 21$. By Proposition \ref{prop:CB}, $\overline{Z}$ satisfies $CB(13)$ and then, by Theorem \ref{thm:CB_implies},
  \[
    21 = \sum_{i=0}^5 Dh_{\overline{Z}}(i) \leq \sum_{i = 9}^{14} Dh_{\overline{Z}}(i).
  \]
  In particular, we deduce that $Dh_{\overline{Z}}(5) = Dh_{\overline{Z}}(9) = 6$ and, by Lemma \ref{lemma:HF_properties}(10), we have $Dh_{\overline{Z}}(i) \geq 6$ for $i \in \{6,7,8\}$. I.e.,
  \[
    \begin{array}{*{21}c}
      &  & $\tiny{0}$ & $\tiny{1}$ & $\tiny{2}$ & $\tiny{3}$ & $\tiny{4}$ & $\tiny{5}$ & $\tiny{6}$ & $\tiny{7}$ & $\tiny{8}$ & $\tiny{9}$ & $\tiny{10}$ & $\tiny{11}$ & $\tiny{12}$ & $\tiny{13}$ & $\tiny{14}$ & \\
      Dh_{\overline{Z}}(i) & : & 1 & 2 & 3 & 4 & 5 & 6 & \geq 6 & \geq 6 & \geq 6 & 6 & 5 & 4 & 3 & 2 & 1 \\
      Dh_{Z_1}(i) & : & 1 & 2 & 3 & 4 & 5 & 6 & 6 & 2 & 1 & - \\
      Dh_{Z'}(14-i) & : & & & & & & & ? & ? & ? & 6 & 5 & 4 & 3 & 2 & 1 & 
    \end{array}
  \]
  Since $Dh_{Z_1}(7) + Dh_{Z_1}(8) = 3$, we conclude that $Dh_{\overline{Z}}(7) + Dh_{\overline{Z}}(8) \geq 9$. Hence,
  \[
    \ell(Z') \geq \sum_{i=7}^{14} Dh_{Z'}(i) \geq 9 + 21 = 30. 
  \]
  Assume now that $Z_1 \cap Z' \neq \emptyset$. As observed in \cite[Remark 3.3]{angelini2020identifiability}, if $Z'' = Z' \smallsetminus (Z' \cap Z_1)$, then there exists a form $f_0 \in S_{13}$ such that $[f] \in \langle \nu_{13}(Z_1) \rangle \cap \langle \nu_{13}(Z'') \rangle$. Indeed, if $Z_1 = \{\ell_1,\ldots,\ell_{30}\}$ and $Z' = \{\ell_1,\ldots,\ell_k,m_{k+1},\ldots,m_r\}$, then,
  \[
    f = \sum_{i=1}^{30}a_i\ell_i^{13} = \sum_{i=1}^kb_i\ell_i^{13} + \sum_{i=k+1}^rb_i m_i^{13}
  \]
  and
  \[
    f_0 = \sum_{i=1}^k(a_i-b_i)\ell_i^{13} + \sum_{i=k+1}^{30}a_i \ell_i^{13} = \sum_{i=k+1}^rb_i m_i^{13}.
  \]
  Therefore, we can proceed with the same argument as before by considering $f_0$ instead of $f$ and $\overline{Z} = Z_1 \cup Z''$ to deduce that $\ell(Z_1) \geq \ell(Z'') \geq 30$. 
\end{proof}
Therefore, we have that $Z_1$ and $Z_2$ have different Hilbert function and regularity by \eqref{eq:HF_example_deg13} and both compute the rank of $f$ by Claim \ref{claim:deg13}. This concludes the proof.
\end{proof}

\section{Explicit Macaulay2 computations}\label{sec:M2}
In this last section, we provide a code in the language of the algebra software Macaulay2 \cite{M2} to construct the examples presented in the previous section. 

~
\subsection{First example. Degree-$10$ ternary forms of rank $22$.}\label{sec:M2_1}
~

\begin{lstlisting}[language = Macaulay2]
i1 : k = QQ;
i2 : S = k[x,y,z];
i3 : -- consider a set of 22 general points
     IZ1 = intersect for i from 1 to 22 list 
                          minors(2,vars(S)||random(QQ^1,QQ^3));
i4 : degree IZ1, for i to 10 list hilbertFunction(i,IZ1)

o4 = (22, {1, 3, 6, 10, 15, 21, 22, 22, 22, 22, 22})

i5 : -- consider the set of points Y = {(0:1:0),(0:0:1),(0:1:1)} on the line x = 0 and the union A = Z1 u Y
     IY = intersect(ideal(x,z), ideal(x,y), ideal(x,y-z));
i6 : IA = intersect(IZ1,IY);
i7 : -- take the degree-6 part of the ideal of A
     G = ideal super basis(6,IA);
i8 : -- take the ideal generated by the restriction of G over the line x = 0 away from the set of points Y
     R = k[y,z];
i9 : G0 = sub(G,R);
i10 : G' = G0 : sub(IY,R);
i11 : -- G' is then defined by three cubics: let C be the cubic curve as image of the line L via the map defined by G'
      phi = map(R,S,first entries mingens G');
i12 : C = ker phi;
i13 : -- We compute the singular locus and we check that it is nodal
      VC = Proj (S/C);
i14 : singVC = singularLocus VC;
i15 : codim singVC, degree singVC

o15 = (2, 1)

i16 : IO = ideal singVC;
i17 : -- We compute the pre-image X of the singular point:
      -- that is a union of two simple points
      T = QQ[a,b,c,y,z, MonomialOrder => Eliminate 3];
i18 : psi = map(T,S,{a,b,c});
i19 : J = psi(IO) + ideal(a-sub(G'_0,T),b-sub(G'_1,T),c-sub(G'_2,T));
i20 : use S;
i21 : IX=saturate(ideal(x)+sub(ideal selectInSubring(1,gens gb J),S));
i22 : dim IX, degree IX, IX == radical IX

o22 = (1, 2, true)

i23 : -- Now, the union A u X is a set of 27 points
      IA' = intersect(IA,IX);
i24 : degree IA', for i to 10 list hilbertFunction(i,IA')

o24 = (27, {1, 3, 6, 10, 15, 21, 26, 27, 27, 27, 27})

i25 : -- Construct a general c.i. of type (7,7)
      B7 = super basis(7,IA');
i26 : C1 = B7 * random(QQ^(numcols B7), QQ^1);
i27 : C2 = B7 * random(QQ^(numcols B7), QQ^1);
i28 : IU = ideal(C1,C2);
i29 : degree IU, for i to 12 list hilbertFunction(i,IU)

o29 = (49, {1, 3, 6, 10, 15, 21, 28, 34, 39, 43, 46, 48, 49})

i30 : -- Consider the linkage to A' through U
      IZ2 = IU : IA';
i31 : degree IZ2, for i to 10 list hilbertFunction(i,IZ2)

o31 = (22, {1, 3, 6, 10, 15, 20, 22, 22, 22, 22, 22})

i32 : -- we compute now the intersection of their linear spans
      perpSpace = ideal super basis(10,IZ1)+ideal super basis(10,IZ2); 
i33 : B10 = basis(10,S);
i34 : f = (B10 * gens ker diff(B10, transpose mingens perpSpace))_0_0;
i35 : -- Z1 and Z2 are apolar to f and have different HFs
      fperp = inverseSystem f; -- computes the apolar ideal of f 
i36 : isSubset(IZ1,fperp), isSubset(IZ2,fperp), 
          netList {for i to 14 list hilbertFunction(i,IZ1), for i to 14 list hilbertFunction(i,IZ2)}
                   +-+-+-+--+--+--+--+--+--+--+--+--+--+--+--+
o36 = (true, true, |1|3|6|10|15|21|22|22|22|22|22|22|22|22|22|)
                   +-+-+-+--+--+--+--+--+--+--+--+--+--+--+--+
                   |1|3|6|10|15|20|22|22|22|22|22|22|22|22|22|
                   +-+-+-+--+--+--+--+--+--+--+--+--+--+--+--+

\end{lstlisting}

~
\subsection{Second example. Degree-$13$ ternary forms of rank $30$.}
~

\begin{lstlisting}[language = Macaulay2]
i1 : k = QQ;	   
i2 : S = k[x,y,z];

i3 : -- consider 12 random points on the rational cubic xz^2 - y^3 = 0
     a = entries random(k^12,k^2);
i4 : A = (p -> {p_0^3, p_0*p_1^2, p_1^3}) \ a;
i5 : IA = intersect ((p -> minors(2,matrix{{x,y,z},p})) \ A);
i6 : degree IA, for i to 5 list hilbertFunction(i,IA)

o6 = (12, {1, 3, 6, 9, 12, 12})

i7 : -- consider a c.i. X of type (6,7) containing A
     B6 = super basis(6,IA);    B7 = super basis(7,IA);    
i9 : F1 = B6 * random(QQ^(numcols B6), QQ^1);
i10 : F2 = B7 * random(QQ^(numcols B7), QQ^1);
i11 : IX = ideal(F1,F2);
i12 : degree IX, for i to 13 list hilbertFunction(i,IX)

o12 = (42, {1, 3, 6, 10, 15, 21, 27, 32, 36, 39, 41, 42, 42, 42})

i13 : -- consider the link to A through X
      IZ1 = IX : IA;
i14 : degree IZ1, for i to 13 list hilbertFunction(i,IZ1)

o14 = (30, {1, 3, 6, 10, 15, 21, 27, 29, 30, 30, 30, 30, 30, 30})

i15 : -- consider a c.i. Y of type (6,10) containing Z1
      B6 = super basis(6,IZ1);    B10 = super basis(10,IZ1);    
i17 : G1 = B6 * random(QQ^(numcols B6), QQ^1);
i18 : G2 = B10 * random(QQ^(numcols B10), QQ^1);
i19 : IY = ideal(G1,G2);
i20 : degree IY, for i to 14 list hilbertFunction(i,IY)

o20 = (60, {1, 3, 6, 10, 15, 21, 27, 33, 39, 45, 50, 54, 57, 59, 60})

i21 : -- consider the link to Z1 through Y
      IZ2 = IY : IZ1;
i22 : degree IZ2, for i to 13 list hilbertFunction(i,IZ2)

o22 = (30, {1, 3, 6, 10, 15, 21, 26, 30, 30, 30, 30, 30, 30, 30})

i23 : -- we compute now the intersection of their linear spans
      perpSpace = ideal super basis(13,IZ1)+ideal super basis(13,IZ2);
i24 : B13 = basis(13,S);
i25 : f = (B13 * gens ker diff(B13, transpose mingens perpSpace))_0_0; 

i26 : -- note that Z1 and Z2 are apolar to f
      -- and have different Hilbert functions and regularity
      fperp = inverseSystem f; -- computes the apolar ideal of f
i27 : isSubset(IZ1,fperp), isSubset(IZ2,fperp), 
          netList {for i to 14 list hilbertFunction(i,IZ1), 
                   for i to 14 list hilbertFunction(i,IZ2)}

                   +-+-+-+--+--+--+--+--+--+--+--+--+--+--+--+
o27 = (true, true, |1|3|6|10|15|21|27|29|30|30|30|30|30|30|30|)
                   +-+-+-+--+--+--+--+--+--+--+--+--+--+--+--+
                   |1|3|6|10|15|21|26|30|30|30|30|30|30|30|30|
                   +-+-+-+--+--+--+--+--+--+--+--+--+--+--+--+
\end{lstlisting}


\bibliographystyle{amsplain}
\bibliography{ACO_Waring_differentHF.bib}

\end{document}